\documentclass[smallextended]{article}       
\usepackage{graphicx}
\usepackage{amsmath, amssymb, amsthm}  
\newcommand{\arc}{\operatorname{arc}}

\newtheorem{theorem}{Theorem}

\newtheorem{lemma}{Lemma}
\newtheorem{corollary}{Corollary}

\begin{document}

\title{Rigid Origami Vertices: Conditions and Forcing Sets\thanks{E. Demaine supported in part by NSF ODISSEI grant EFRI-1240383 and NSF Expedition grant CCF-1138967. D. Eppstein supported by NSF grant CCF-1228639 and  ONR/MURI grant N00014-08-1-1015.  T. Hull supported by NSF ODISSEI grant EFRI-1240441.  R. J. Lang supported in part by NSF ODISSEI grants EFRI-1240441, EFRI-1240417, EFRI-1332249, and EFRI-1332271.
T. Tachi supported by the Japan Science and Technology Agency Presto program.}
}


\author{Zachary Abel\thanks{MIT, {\tt \{zabel, edemaine, jasonku\}@mit.edu}}     \and
            Jason Cantarella\thanks{UGA, {\tt jason@math.uga.edu}} \and
            Erik D. Demaine\footnotemark[2] \and
            David Eppstein\thanks{Univ. of California, Irvine, {\tt david.eppstein@gmail.com}} \and
            Thomas C. Hull\thanks{Western New England Univ., {\tt thull@wne.edu}} \and
            Jason S. Ku\footnotemark[2] \and
            Robert J. Lang\thanks{Lang Origami, CA, {\tt robert@langorigami.com}} \and
            Tomohiro Tachi\thanks{Univ. of Tokyo, {\tt tachi@idea.c.u-tokyo.ac.jp}}
}




\maketitle

\begin{abstract}
We develop an intrinsic necessary and sufficient condition for single-vertex origami crease patterns to be able to fold rigidly.  We classify such patterns in the case where the creases are pre-assigned to be mountains and valleys as well as in the unassigned case.  We also illustrate the utility of this result by applying it to the new concept of minimal forcing sets for rigid origami models, which are the smallest collection of creases that, when folded, will force all the other creases to fold in a prescribed way.  
\end{abstract}

\section{Introduction}
\label{intro}

Rigid origami, where stiff panels are folded along hinged creases, has numerous applications in kinetic architecture \cite{Tachi2}, deploying solar panels into outer space \cite{Miura,Schenk}, and robotics \cite{Balkcom}, to name a few.  Such applications have motivated much research into the mathematics and mechanics of rigid folding over many years  \cite{Huff,Miura1,Tachi1}.  However, a complete theory of the mathematics of rigid origami is not yet complete.

Whenever we apply or design rigid origami systems, we encounter a common important problem called \emph{rigid foldability}, which is a problem of judging whether a rigid origami on given crease pattern can continuously transform from a planar state to an intermediate folded state.
Rigid foldability has been represented using extrinsic parameters of the folded state, e.g., the existence of a set of fold angles satisfying compatibility conditions \cite{belcastro,Kawa} or the existence of intermediate state \cite{Streinu,Tachi0}. However, completely intrinsic representation of rigid foldability, i.e., conditions of crease patterns to be rigidly foldable, is not yet characterized even for a single-vertex origami.
In this paper, we give a necessary and sufficient condition that characterizes rigidly foldable single-vertex crease patterns. We believe that this serves as the basic theorem for rigid-foldable origami, as an equivalent of Kawasaki's and Maekawa's theorems \cite{Dem,Hull} for flat-foldable origami.

In addition, we showcase the utility of our result by applying it to the new concept of a {\em forcing set} for rigidly foldable vertices. A forcing set is a subset of the creases whose folded state uniquely determines the folded state of the remaining creases.  Finding forcing sets for origami crease patterns is useful for the emerging field of self-folding mechanics, where programming only a subset of the creases to self-fold and achieve a desired folded state can be more economical \cite{Hull5}. We use our single-vertex rigid foldability result to characterize the possible sizes of minimal forcing sets in such vertex folds. 

This work was developed by the authors during the 29th Winter Workshop on Computational Geometry at the Bellairs Research Institute of McGill University, March 2014.

\section{Rigid Vertex Model}
\label{sec:1}

We begin by defining a model for rigid origami similar in concept to that of \cite{Balkcom,belcastro,Dem,Tachi3}.    Given a region $A\subseteq\mathbb{R}^2$ (often referred to as the {\em paper}), a {\em crease pattern} $C$ is a planar straight-line graph embedding that partitions $A$ into vertices, edges, and faces.  A {\em rigid origami} (or {\em rigid folding}) $f:A\rightarrow \mathbb{R}^3$ is an injective, continuous map that is non-differentiable only on $C$ and an isometry on each face of $C$.  The vertices and edges of $C$ are called the {\em vertices} and {\em crease lines}, respectively, of the rigid folding.  At each crease line $c$ we will refer to the {\em folding angle} of $c$ to be the signed angle by which the two faces adjacent to $c$ deviates from a flat plane.
We note that by requiring $f$ to be injective, we require all the folding angles at the creases to be unequal to $\pi$ or $-\pi$, although we can make them as close to these angles as we wish.  Dealing with completely flat-folded creases (folding angles equal to $\pm\pi$) involves combinatorial issues that will not be relevant in this paper.  (For more information, see \cite{Dem}.)

A {\em single-vertex} rigid origami is one with more than two creases where the all of the crease lines intersect at a single point in the interior of $A$.  After single-fold origami, single-vertex crease patterns are the simplest types of origami examples, and the combinatorics of single-vertex crease patterns that fold flat (with no attention to rigidity) have been extensively explored \cite{Hull,Hull1,Chang}.

When we fold a rigid origami model, we must assign the creases to be either {\em mountain} (convex) or {\em valley} (concave).  Formally, we define a {\em mountain-valley (MV) assignment} of a crease pattern $C$ to be a function $\mu:C\rightarrow \{M,V\}$.  When the paper is unfolded, the folding angle at each crease will be zero.  As the paper is rigidly folded, a valley crease will have a positive folding angle, while a mountain crease will have a negative folding angle.

The folding angles at the creases give us a way to describe a {\em continuously parameterized family} of rigid origami folds from the same crease pattern $C$ that track how a crease pattern can be rigidly folded from the flat, unfolded state to a folded configuration where every crease has a non-zero folding angle.  In practice, we would model the rigid origami like a robotic arm linkage and determine the number of degrees of freedom of the crease pattern (as done in \cite{Balkcom} and \cite{Tachi1}) to see how many parameters we would need.  In this paper, however, we will not need to delve into that amount of detail.  Note, however, that the existence of a rigid origami for a given crease pattern and MV assignment implies that there exists a path in the parameter space to the flat, unfolded state because the configuration space of single-vertex rigid origami motions has been proven to be connected \cite{Streinu}.

Our goal is to develop conditions on a single-vertex crease pattern $C$ with a MV assignment $\mu$ to guarantee that it will be able to rigidly fold, say from the unfolded state continuously to a folded state in which every crease in $C$ has non-zero folding angle.

\begin{figure}
\centerline{\includegraphics[scale=.4]{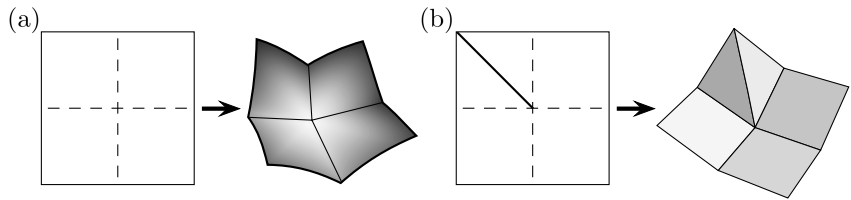}}
\caption{(a) A degree 4 vertex with all valley creases cannot fold rigidly, even a little bit. (b) With a mountain crease added, however, it will rigidly fold.}\label{fig1}
\end{figure}

Figure \ref{fig1}(a) shows an example of a vertex that will not rigidly fold, namely a degree four vertex where every crease is a valley.  If all four valley creases are folded, even by a small amount, the regions of paper between them will be forced to curve and thus be non-rigid. (See \cite{Hull2,Miura1} for a formal proof of this.)  However, if we add a mountain crease to this vertex between any two of the valleys, it will be able to fold rigidly.  An example of this is shown in Figure \ref{fig1}(b).   

\section{Conditions for a Vertex to Fold Rigidly}
\label{sec:2}

Let $\angle(c,d)$ denote the oriented angle going counterclockwise between two creases $c$ and $d$ that meet at a vertex.

Given a single-vertex crease pattern $C$ together with a MV assignment $\mu:C\rightarrow \{M,V\}$, we say that the pair {\em $(C,\mu)$ has a  tripod} (or a {\em mountain tripod}) if there exist three creases $c_1, c_2, c_3\in C$ in counterclockwise order, not necessarily contiguous (e.g., there could be other creases between $c_1$ and $c_2$, etc.) with the property that $\mu(c_i)=M$ for $i=1, 2, 3$ and $0<\angle(c_i, c_{i+1}) <\pi$ for $i=1,2,3$ (mod 3).  A {\em valley tripod} is a tripod but with $\mu(c_i)=V$ for $i=1, 2, 3$. See Figure~\ref{fig2}.

We also say that $(C,\mu)$ contains a {\em mountain (resp. valley) cross} if there exists four creases in counterclockwise order $c_1, c_2, c_3, c_4\in C$, not necessarily contiguous, where $c_1$ and $c_3$ form a straight line, as do $c_2$ and $c_4$, and $\mu(c_i)=M$ (resp. $\mu(c_i)=V$) for $i=1, 2, 3, 4$.

We say that $(C,\mu)$ contains a {\em bird's foot} if it contains either a tripod or a cross together with another crease that has the opposite MV-parity than the tripod/cross.  Figure~\ref{fig2} shows bird's feet made from a mountain tripod (left, creases $c_1$--$c_4$) and a mountain cross (right, creases $c_1$--$c_5$).

\begin{figure}
\centerline{\includegraphics[scale=.7]{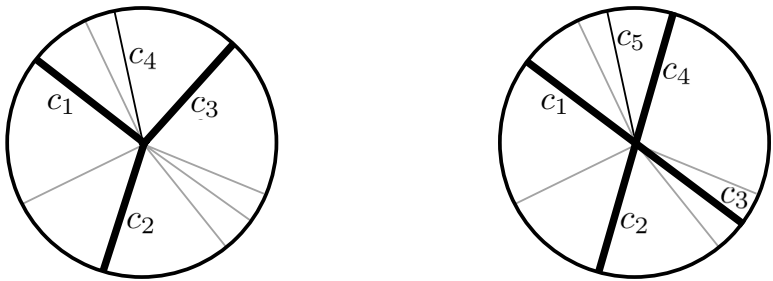}}
\caption{Single-vertex crease patterns with a mountain tripod (left) and a mountain cross (right). These are also bird's feet.}\label{fig2}
\end{figure}

\subsection{Main theorem}\label{sec:3.1}

\begin{theorem}\label{thm1}
A single-vertex crease pattern $(C,\mu)$ can be continuously parameterized in a family of rigid origami folds if and only if $(C,\mu)$ contains a bird's foot.
\end{theorem}

We will prove Theorem \ref{thm1} by considering our vertex to be at the center of a sphere sufficiently small so that every crease at the vertex intersects with the sphere.  Our rigid origami vertex will then intersect our sphere along arcs of great circles, and so an equivalent way to think about our rigid origami vertex as a {\em closed spherical linkage} that forms a non-intersecting closed loop on the sphere.  Such a linkage will have vertices $c_i$ corresponding to the creases of our rigid origami, and we can establish that the angle $\alpha_i$ on the sphere at vertex $c_i$ is the supplement of the folding angle at crease $c_i$.  This also means  that we can carry the MV assignment $\mu$ from the creases to the vertices of our spherical linkage.

Let $\arc(c_i, c_{i+1})$ denote the length of the spherical arc between consecutive vertices $c_i$ and $c_{i+1}$ in a spherical linkage. By the assumption that the folding pattern has more than two creases, each $\arc(c_i, c_{i+1})$ is less than $\pi$.

\begin{lemma}\label{lem1}
Given a closed spherical linkage $C = \{c_1, ..., c_k\}$ where the $c_i$ are the vertices, in order, embedded on a sphere in some configuration (i.e., with angles $\alpha_i\notin\{0, \pi\}$ at each vertex $c_i$ for $i=1, ..., k$ and with  $0<\arc(c_i, c_{i+1})<\pi$ for each arc), we may shrink the length of any $\arc(c_i, c_{i+1})$ by some small amount and it will only change the angles $\alpha_{i-1}$ and $\alpha_i$ by a small amount without changing the MV parity of the vertices. 
\end{lemma}

\begin{proof}
Draw a circle $C_a$ on the sphere whose center is the midpoint of $\arc(c_i, c_{i+1})$ such that the circle contains the points $c_i$ and $c_{i+1}$ (i.e., has the proper radius).  Note that $C_a$ will lie in a hemisphere because $\arc(c_i,c_{i+1})<\pi$.   Draw a circle $C_i$ centered at $c_{i-1}$ such that the circle contains $c_i$. (See Figure~\ref{fig2.5}; the assumption that $\arc(c_{i-1}, c_i)<\pi$ prevents $C_i$ from being degenerate.) Then $C_i$ represents the region on the sphere where vertex $c_i$ can be if we rotate it about $c_{i-1}$.  We claim that $C_a$ and $C_i$ are not tangent, for if they were then we'd have $\alpha_{i}=\pi$ and the vertex $c_{i}$ would not be folded.  Thus $C_a$ and $C_i$ overlap, and so we can shrink the length of $\arc(c_i, c_{i+1})$ a little bit and rotate $c_i$ around $c_{i-1}$ a little bit to have vertex $c_i$ still be connected to the rest of the linkage.   A similar argument will work for the vertex $c_{i+1}$.  
\end{proof}

\begin{figure}
\centerline{\includegraphics[scale=.7]{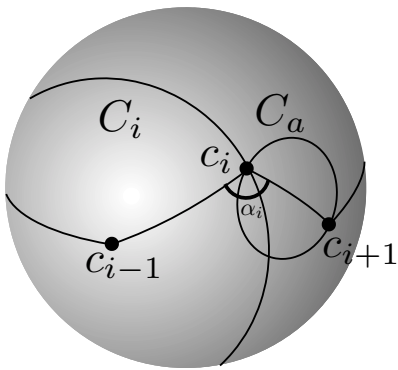}}
\caption{Illustration of the spherical linkage for Lemma \ref{lem1}.}\label{fig2.5}
\end{figure}

The next lemma is interesting in its own right and will be fundamental in the proof of Theorem \ref{thm1} as well as the result on forcing sets in Section~\ref{sec:4}.

\begin{lemma}\label{lem2}
Single-vertex continuously parameterized rigid foldability is closed under adding creases.
\end{lemma}

Here by ``closed under adding creases" we mean that if a single-vertex crease pattern $C$ is rigidly foldable (by some non-zero dihedral angles of all the creases) then so is $C\cup\{c'\}$ where $c'$ is some new crease added to the single vertex.

\begin{proof}
Let $C=\{c_1, ..., c_k\}$ be a single-vertex crease pattern that is rigidly foldable.  
Suppose the new crease $c'$ that we are adding is between creases $c_i$ and $c_{i+1}$.  Consider the vertex to be at the center of a sphere and fold $C$ into a rigid configuration.  We know $\angle(c_i,c_{i+1})<\pi$, for if otherwise then the other angles would form a shortest path on the sphere from the endpoint of $c_i$ to the endpoint of $c_{i+1}$, meaning they couldn't be folded any closer together, and thus the other creases couldn't be in a valid folded configuration.  Now use Lemma \ref{lem1} to shrink $\arc(c_i, c_{i+1})$ a little bit.  Then replace $\arc(c_i, c_{i+1})$ with a small ``tent" (if we want a M) or a ``trough" (for a V) at the position where we want $c'$ {\em and} so that the original  $\arc(c_i, c_{i+1}) = \arc(c_i, c') + \arc(c', c_{i+1})$.  See Figure \ref{fig3}.  Since we first decreased the length of $\arc(c_i, c_{i+1})$, we will have that the folding angle at $c'$ will not be zero, and so the crease $c'$ will be folded and all the other creases are at most only changed slightly from their previous folding angles.  The result is the new crease pattern $C\cup\{c'\}$ that is in a rigidly-folded configuration.  Furthermore, since this demonstrates a rigidly-folded state of the single-vertex crease pattern $C\cup\{c'\}$, we know it can be continuously parameterized to the unfolded state by \cite{Streinu}. 
\end{proof}

\begin{figure}
\centerline{\includegraphics[scale=.6]{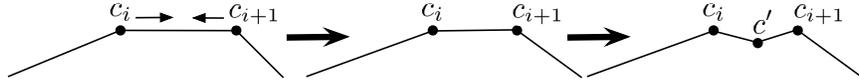}}
\caption{Adding a crease to a single-vertex crease pattern, as witnessed in the proof of Lemma 2.}\label{fig3}
\end{figure}

\begin{lemma}\label{lem3}
Degree 4 and 5 vertices that are bird's feet will rigidly fold.
\end{lemma}

\begin{proof}
Suppose the bird's foot contains a tripod with three valley creases $c_1$, $c_2$, and $c_3$ and one mountain crease $c_4$, which we may assume is between $c_1$ and $c_3$ like the one shown in Figure~\ref{fig4}(a).  We aim to build a closed spherical linkage for this vertex and thus prove that it has a rigidly folded state.  Place $c_1$ and $c_3$ on the equator of a sphere so that $\arc(c_1, c_3)<\alpha_3+\alpha_4<\pi$.   Then the arcs $c_1c_2$ and $c_2c_3$ can be drawn on the northern hemisphere of the sphere so that $c_1c_2c_3$ forms a spherical triangle (see Figure~\ref{fig4}(b)) whose exterior angles $\theta_1, \theta_2, \theta_3 > 0$. (Here, $\theta_2$ is the fold angle of crease $c_2$.)    

Consider the arcs $c_3c_4$, and $c_4c_1$ to be drawn on the sphere. There will be two candidate positions for $c_4$ to do this because $\arc(c_1, c_3)<\alpha_3+\alpha_4$. We chose the one in the northern hemisphere. Label the interior angles of triangle $c_1c_3c_4$ with $\phi_1, \phi_4, \phi_3 > 0$. (Here, $\phi_4 - \pi$ is the fold angle of crease $c_4$.)

The fold angles of $c_1$ and $c_3$ are then $\theta_1 + \phi_1 > 0$ and $\theta_3 + \phi_3  > 0$, respectively; see Figure~\ref{fig4}(c).
The motions of $c_1c_2c_3$ and $c_1c_3c_4$ are both continuous from a flat state. This means that we can assume that $\theta_1, \theta_3$ and $\phi_1, \phi_3$ are all small enough so that $\theta_1 + \phi_1$ and $\theta_3 + \phi_3$ do not exceed $\pi$, thus ensuring that $c_1$ and $c_3$ will be valley creases.

Therefore this produces a folding motion with $c_1$, $c_2$, $c_3$ being valleys and $c_4$ being a mountain.

For cross case, just replace triangle $c_1c_2c_3$ with a $c_1c_2c_2'c_3$ spherical quadrilateral, and the rest follows the same argument.  This gives us a rigidly-folded instantiation of our bird's foot, which we know can then be rigidly unfolded by \cite{Streinu}. 

\end{proof}

\begin{figure}
\centerline{\includegraphics[scale=.5]{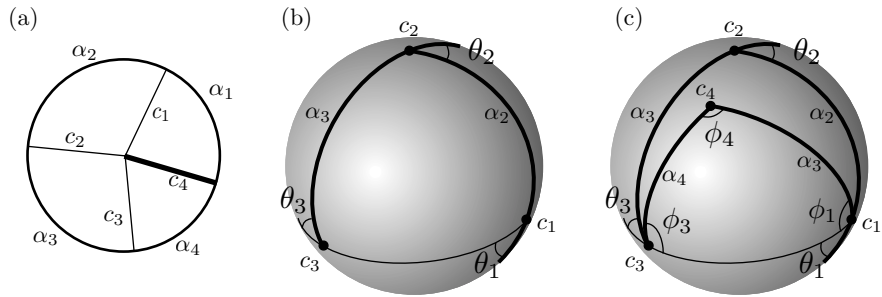}}
\caption{Three creases of a tripod bird's foot make a spherical triangle when folded, and then the fourth crease may be added.}\label{fig4}
\end{figure}

\vspace{.2in}

\begin{proof}[of Theorem \ref{thm1}]  
We start with the $\Leftarrow$ direction.  We know by Lemma \ref{lem3} that a degree 4 single-vertex crease pattern that is a bird's foot, with tripod, will rigidly fold, as will a degree 5 vertex that is a bird's foot with cross.  By Lemma \ref{lem2} if we add creases to this it will still rigidly fold.  Therefore any single-vertex crease pattern that contains a bird's foot will rigidly fold in a continuous parameterized family.

Of course, the other direction is harder.  We assume $(C,\mu)$ will fold rigidly (by some small amount) and suppose that no bird's foot (tripod or cross) exist in $(C,\mu)$.  Let us also assume that $(C,\mu)$ is not all Ms (or all Vs) because in those cases there are standard ways to show that it can't rigidly fold even a little while using all of the creases (e.g., the Gauss map makes a spherical polygon with non-zero area; see \cite{Huff} or \cite{Miura1}).

So what must $(C,\mu)$ look like?  Well, the set of all mountains in $C$ must be confined to a half-plane (or a semicircle, if we think of $C$ as being drawn on a disk).  This is because no tripods and no crosses implies all the mountains are contained in a half-circle.  The same is true for all of the valley creases in $C$.  

\begin{figure}
\centerline{\includegraphics[scale=.7]{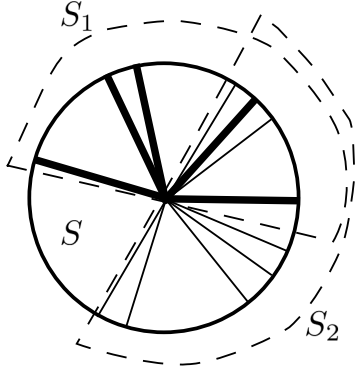}}
\caption{Overlapping semicircles $S_1$ and $S_2$, and the sector $S$ that will remain fixed in the folding of the vertex.}\label{fig:overlap}
\end{figure}

The half-circle of all Ms could overlap with the half-circle of all Vs, however.  In fact, we can always pick two semicircles $S_1$ and $S_2$ such that $S_1$ contains all the Ms, $S_2$ contains all the Vs, and $S_1\cap S_2\not=\emptyset$.  (Even if the Ms and Vs split the disk evenly, we can modify one of the semicircles to make them have overlap.)  This will mean that there is a sector of the paper disk, call it $S$, that is neither completely in $S_1$ nor $S_2$.  (See Figure \ref{fig:overlap}.)  The sector $S$ defines a plane, and we imagine this plane being fixed, say in the equator disk of a sphere $B$, while we fold the rest of the disk inside $B$.

The sector $S$ is bounded by a M crease and a V crease.  Let $a$ (resp. $b$) be the endpoint on the circle of the M (resp. V) crease bounding $S$, and let $\alpha$ be the angle between $a$'s mountain crease and $b$'s valley crease.  Let $a'$ (resp. $b'$) be the antipodal point of $a$ (resp. $b$).  Then $a'$ is on the boundary of the semicircle $S_1$, and so between $b$ and $a'$ there are only valley creases.  Similarly, between $a$ and $b'$ there are only mountain creases.  Our goal is to show that $a'$ and $b'$ get further apart in the folding, which is a contradiction.  (See Figure~\ref{fig5} for visuals on this.)

We track regions on the surface of the sphere $B$ to which $a'$ and $b'$ can move under the rigid folding.  Denote $C_p(r)$ to be a circle, with interior, drawn on the sphere with center $p$ on the sphere and arc-radius $r$.  That is,
$C_p(r) = \{x\in B\ :\ \arc(x,p)\leq r\}$.
Then the region where point $a'$ can be folded is precisely $C_b(\pi-\alpha)$; the boundary of $C_b(\pi-\alpha)$ is achieved by rotating $a'$ about $b$-$b'$ line, which is the same as folding the crease whose endpoint is $b$, and folding any of the other valley creases between $a'$ and $b$ will only move $a'$ inside $C_b(\pi-\alpha)$. Similarly, the image of the point $b'$ under the rigid folding (while keeping sector $S$ fixed) is $C_a(\pi-\alpha)$.

\begin{figure}
\centerline{\includegraphics[scale=.7]{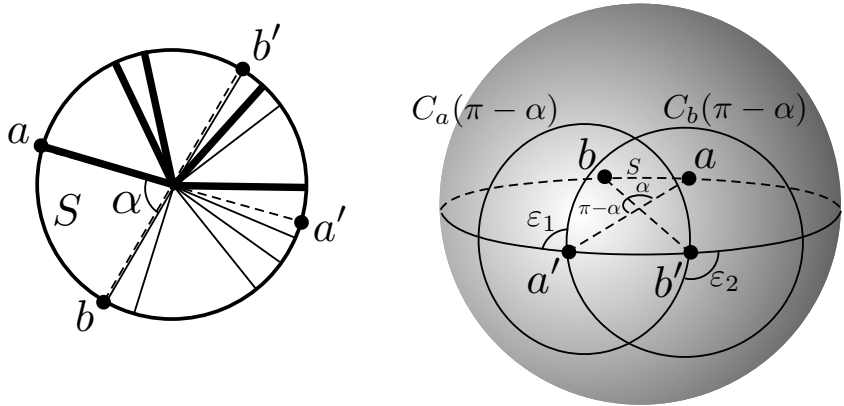}}
\caption{In the proof of Theorem~\ref{thm1}, when the creases are all folded a little, $a'$ and $b'$ can only move into the regions $\varepsilon_1$ and $\varepsilon_2$, respectively.}\label{fig5}
\end{figure}

Now, if all the valley creases between $a'$ and $b$ are folded, the point $a'$ will move above the equator of $B$ in the region marked $\varepsilon_1$ in Figure~\ref{fig5}.  Similarly, if all the mountain creases between $b'$ and $a$ are folded, the point $b'$ will move below the equator in the region labeled $\varepsilon_2$ in Figure~\ref{fig5}.  But the shortest distance between the regions $\varepsilon_1$ and $\varepsilon_2$ is $\arc(a',b')$ which is only achieved along the equator when the crease pattern $C$ is in the unfolded position.  Thus when the creases are folded, the points $a'$ and $b'$ must increase their distance from each other, which is impossible because it would require the paper to rip. 
\end{proof}

An interesting corollary follows from Theorem \ref{thm1}.  We say that a single-vertex crease pattern {\em pops up} when folded if all the creases fold to the lower half-sphere when we place the vertex at the center of a sphere with the unfolded paper on the equatorial plane.  We say that a single vertex {\em pops down} if the creases all fold to the upper half-sphere.

\begin{corollary}\label{cor1}
If a single-vertex crease pattern has a M tripod or cross then it can be folded to pop up.  If a single-vertex crease pattern has a V tripod or cross then it can be folded to pop down.  
\end{corollary}

The proof of this simply follows from the construction given in Lemma~\ref{lem2}.  

A  consequence of Corollary~\ref{cor1} is that if a single vertex has both a M  and a V tripod (or cross), then it can be made to snap from a pops up folded state to a pops down one.  This characterizes formally a phenomenon seen in practice when rigidly folding single-vertex crease patterns.

\subsection{Vertices without a MV assignment}\label{sec:3.2}

It is possible to make a more general version of Theorem~\ref{thm1} that looks at a single-vertex crease pattern $C$ with no MV assignment $\mu$ associated to it.  In other words, we are asking the question, ``Just by looking at the arrangements of creases meeting at a vertex, can we determine if it could rigidly fold in some way?"  

To do this, we need to augment our definitions slightly.  By an {\em unspecified cross} in a single-vertex crease pattern $C$ we mean cross without the mountain-valley assignment $\mu$.  That is, the unspecified cross is four creases $c_1, c_2, c_3, c_4\in C$ such that $c_1$ and $c_3$ form a straight line, as do $c_2$ and $c_4$.
We also say that a single-vertex crease pattern $C$ drawn on a domain $A$ is {\em rigidly foldable} if there exists an injective, continuous, piecewise isometry $f:A\rightarrow\mathbb{R}^3$ that is a rigid origami with crease pattern $C$.

\begin{theorem}\label{thm1.5}
A single-vertex crease pattern $C$ is rigidly foldable if and only if $C$ has at least four creases, all sector angles of $C$ are strictly less than $\pi$, and $C$ is not an unspecified cross.
\end{theorem}

\begin{proof}
If $C$ is rigidly foldable, then Theorem \ref{thm1} gives us that $C$ contains a bird's foot.  This $C$  cannot be an unspecified cross, and the bird's foot sector angles are all less than $\pi$, implying that the sector angles for all of $C$ are less than $\pi$.  

In the other direction, suppose we have a single-vertex crease pattern $C$ with no MV assignment and at least four creases, not an unspecified cross, and all sector angles less than $\pi$. Let $W$ be the wedge between consecutive creases with the largest sector angle. It is possible for there to be another wedge diametrically opposite $W$, with the same angle, so that the four creases bounding this wedge form an unspecified cross, but in this case in order for the whole pattern to not be an unspecified cross there must be an additional crease somewhere; labeling the cross mountain and the additional crease valley creates a bird's foot and fulfills the requirements of a rigidly foldable vertex. On the other hand, if the creases do not form an equal wedge diametrically opposite $W$, then there is a crease within that opposite wedge (since $W$ has the largest sector angle) that, together with $W$, forms a tripod. Again, we can assign this tripod to be a mountain fold and assign any other crease to be a valley fold and thus satisfy Theorem \ref{thm1}. 
\end{proof}

\section{Forcing Sets in Rigid Vertex Folds}
\label{sec:4}

Given a rigidly foldable (at least by a small amount) vertex $C_n$ of degree $n$ and a MV assignment $\mu:C_n\rightarrow \{M,V\}$, let $f(C_n)$ denote a {\em minimal forcing set} for $C_n$.  That is, $f(C_n)$ is a subset of the creases such that the only possible MV assignment for $C_n$ that agrees with $\mu$ on $f(C_n)$ is $\mu$ itself, and no other forcing set of smaller size exists.

\begin{theorem}\label{thm2}
Using the above notation, we have that
$$n-3\leq |f(C_n)| \leq n\mbox{ for }n\geq 6,$$
$$2\leq |f(C_5)|\leq 4,\mbox{ and}$$
$$1\leq |f(C_4)|\leq 2.$$
\end{theorem}
 
\begin{proof} 
Let $n\geq 6$ and suppose for the sake of contradiction that there exists a forcing set $F=f(C_n)$ with $|F|\leq n-4$.  Then the complement $\overline{F}$ is size at least 4.  
 
We consider cases where either $F$ or $\overline{F}$ contain parts of bird's feet (or not).  Without loss of generality we may assume that any tripods or crosses that are part of such bird's feet are mountains.
 
 \vspace{.1in}
 
 {\em Case 1:}  $F$ contains a tripod or cross.  
 
 Then one of the creases, $c$, in $\overline{F}$ is in between the creases of the tripod/cross, and we can let this crease be a V.  Then using the same method described in the proof of the rigid classification theorem (Theorem \ref{thm1}), we can make the other creases in $\overline{F}$ be any combinations of Ms or Vs we want.  That is, we consider the crease pattern $C_n \setminus (\overline{F}\setminus \{c\})$ and rigidly fold this (which is possible because it contains a bird's foot) and then add the creases in $\overline{F}\setminus \{c\}$ one-at-a-time using the process described in the proof of Lemma \ref{lem2}, assigning these three creases to be whatever combination of mountains or valleys we like.  This contradicts the fact that $F$ is forcing.
 
 \vspace{.1in}
 
 {\em Case 2:}  $\overline{F}$ contains a tripod (say all Ms).  
 
 Then let the other crease in $\overline{F}$ be a V to make a bird's foot.  Then we can construct the rigid folding of $C_n$ and the MV assignment $\mu$ by first folding $\overline{F}$ (leaving the creases in $F$ unfolded, with folding angles zero) and then folding the creases in $F$, one-at-a-time, according to $\mu$ using the method from the proof of Lemma \ref{lem2}.
 
 However, we could do the same thing except starting with the tripod in $\overline{F}$ being all Vs and the other crease in $\overline{F}$ being a M.  This would give us two different rigidly foldable MV assignments for $C_n$ that agree on the creases in $F$, contradicting with $F$ being forcing.

 \vspace{.1in}
 
 \begin{figure}
\centerline{\includegraphics[scale=.6]{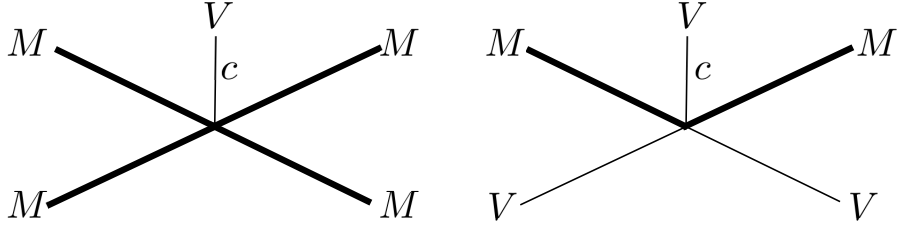}}
\caption{The cross in the complement of $F$ can be folded in two different ways, either as a mountain cross or as a valley tripod.}\label{fig6}
\end{figure}
 
 {\em Case 2.5:}  $\overline{F}$ contains a cross (say all Ms).  
 
 Then $F$ must contain a V ($C_n$ can't be all Ms), so call this crease $c$.  This crease $c$ must be between two of the legs of the cross, and now we can see that the cross can be folded in two different ways, as shown in Figure \ref{fig6}.  We then obtain a contradiction using the same method as in Case 2.
 
 \vspace{.1in}
 
 {\em Case 3:}  Neither $F$ nor $\overline{F}$ contains a tripod or a cross.
 
 Since $C_n$ folds rigidly with $\mu$, it must contain a cross or a tripod as part of some bird's foot.  Also, $\overline{F}$ must live in a half-plane (because if it didn't, then it would contain a potential tripod or cross).
 
 So there exists a crease of $F$ that creates a....
 
 (a) tripod with some of $\overline{F}$.  But then there exists a crease left over in $\overline{F}$ since $|\overline{F}|=4$.  This left over crease can either be a M or a V, contradicting $F$ being forcing.
 
 (b) cross, but only two of the creases in $\overline{F}$ can be in this cross (since $\overline{F}$ lives in a half-plane), leaving at least one crease left on $\overline{F}$ to be a M or a V.  This contradicts  $F$ being forcing.  
 
Cases 1 thru 3 show that $|f(C_n)|\geq n-3$ for $n\geq 6$.  Figure~\ref{fig9} shows examples where $|f(C_n)|=n$ and $|f(C_n)|=n-3$, making our bounds sharp.

 \begin{figure}
\centerline{\includegraphics[scale=.7]{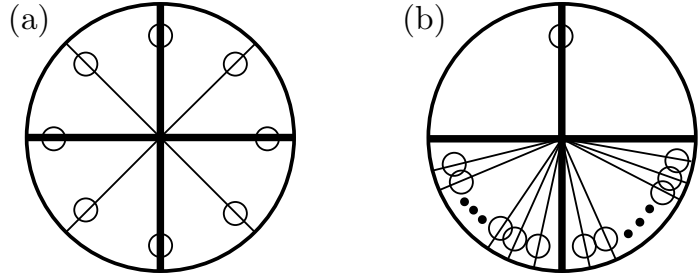}}
\caption{(a) If $n$ is even, the angles between the creases are all equal, and we alternate Ms and Vs, then $|f(C_n)|=n$, since any crease not in the forcing set could be flexed.  (b) A configuration where we have $|f(C_n)|=n-3$, since only one crease from the mountain cross needs to be in the forcing set.}\label{fig9}
\end{figure}
 
When $n<6$ the bounds for the size of the forcing set can be a little different.  The smallest case, $n=4$, can have $|f(C_4)|=1$ by choosing one crease from a mirror-symmetric tripodal bird's foot tripod.  Or we could have $|f(C_4)|=2$ by picking two legs of an asymmetric tripod.  (These are illustrated in Figure~\ref{fig10}(a).)  Since any $n=4$ rigidly foldable singlevertex crease pattern must be a tripodal bird's foot, this covers all cases when $n=4$.

In the $n=5$ case, our crease pattern $(C,\mu)$ could be a tripod bird's foot with an extra crease or a cross bird's foot.  In the former case, if all the angles between creases are equal, we can get away with $|f(C_5)|=4$ , as shown in Figure~\ref{fig10}(b).  If we have a cross, like the mountain cross shown in Figure~\ref{fig10}(b), we need one of the mountain cross legs and the valley crease in our forcing set, giving us $|f(C_5)|=2$.  
 
 This completes the proof of Theorem~\ref{thm2}. 

 \begin{figure}
\centerline{\includegraphics[scale=.5]{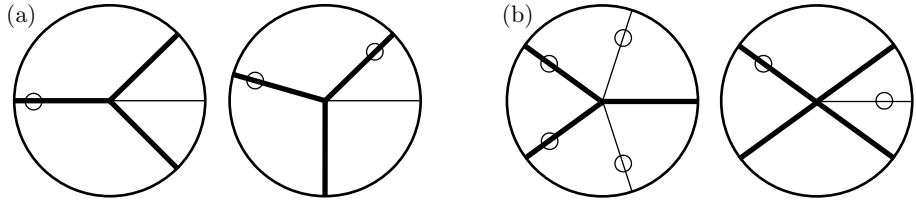}}
\caption{Different kinds of forcing sets when (a) $n=4$ and (b) $n=5$.}\label{fig10}
\end{figure}

\end{proof}
 
 \section{Conclusion}
\label{sec:5}
 
We have proven a necessary and sufficient condition for a single-vertex crease pattern to be able to fold rigidly from an unfolded state.  This condition depends only on the intrinsic geometry of the given crease pattern, as opposed to knowing the specific folding angles of the creases.  It is therefore very easy to check and can serve as a foundation-level result for rigid-foldability in the same way that Kawasaki's and Maekawa's Theorems are for flat-foldability.
We have also demonstrated the utility of this result by applying it to determine bounds on the size of minimal forcing sets of rigidly foldable single-vertex crease patterns.  

Finding a similar rigid-foldability result for multi-vertex crease patterns is an open problem.
The multiple vertex problem is more difficult than just checking the local conditions because we would also need to make the folding angles be consistent from one vertex to adjacent vertices.  As in the flat-foldability problem, characterizing global rigid-foldability seems to be difficult.


\bibliographystyle{plain}       
\bibliography{ACDEHKLT.bib}   

\end{document}